\documentclass[12pt]{article}
\usepackage{amsmath}
\usepackage{amssymb}
\usepackage{amsthm}
\usepackage{graphicx}
\usepackage{graphics}
\usepackage{amscd}
\usepackage{t1enc}
\usepackage[latin2]{inputenc}
\usepackage[usenames,dvipsnames]{color}
\usepackage{paralist}

\newcommand{\cut}[1]{}

\newtheorem{thm}{Theorem}[section]

\newtheorem{theorem}[thm]{Theorem}
\newtheorem{prop}[thm]{Proposition}
\newtheorem{proposition}[thm]{Proposition}
\newtheorem{lemma}[thm]{Lemma}
\newtheorem{remark}[thm]{Remark}

\newtheorem{definition}[thm]{Definition}

\newtheorem{result}[thm]{Result}

\newcommand{\cL}{\mathcal{L}}
\newcommand{\cB}{\mathcal{B}}
\newcommand{\cC}{\mathcal{C}}
\newcommand{\cV}{\mathcal{V}}
\newcommand{\cK}{\mathcal{K}}
\newcommand{\cH}{\mathcal{H}}
\newcommand{\cS}{\mathcal{S}}
\newcommand{\PG}{\mathrm{PG}}
\newcommand{\tb}{\tilde{\cB}}
\newcommand{\wt}[1]{\widetilde{#1}}
\newcommand{\wh}[1]{\widehat{#1}}

\newcommand{\GF}{\mathrm{GF}}
\newcommand{\UCN}{\bar\chi}
\newcommand{\qbinom}[2]{\left[\!\begin{smallmatrix}#1\\#2\end{smallmatrix}\!\right]_q}
\newcommand{\dqbinom}[2]{\begin{bmatrix}#1\\#2\end{bmatrix}_q}

\begin{document}

\title {On the upper chromatic number and multiple blocking sets of $\PG(n,q)$
 \vspace{3mm}}
\author  {Zolt\'an L.\ Bl\'azsik \and Tam\'as H\'eger \and Tam\'as Sz\H{o}nyi}
\date{\scriptsize Latest update on \today}
\maketitle

\begin{abstract}
We investigate the upper chromatic number of the hypergraph formed by the points and the $k$-dimensional subspaces of $\mathrm{PG}(n,q)$; that is, the most number of colors that can be used to color the points so that every $k$-subspace contains at least two points of the same color. Clearly, if one colors the points of a double blocking set with the same color, the rest of the points may get mutually distinct colors. This gives a trivial lower bound, and we prove that it is sharp in many cases. Due to this relation with double blocking sets, we also prove that for $t\leq \frac38p+1$, a small $t$-fold (weighted) $(n-k)$-blocking set of $\mathrm{PG}(n,p)$, $p$ prime, must contain the weighted sum of $t$ not necessarily distinct $(n-k)$-spaces.
\end{abstract}

\section{Introduction and results}
Throughout the paper, let $\cH$ denote a hypergraph with point set $V$ and edge set $E$. A strict $N$-coloring $\cC$ of $\cH$ is a coloring of the elements of $V$ using exactly $N$ colors; in other words, $\cC=\{C_1,\ldots,C_N\}$ is a partition of $V$ where each $C_i$ is nonempty ($1\leq i\leq N$). Given a coloring $\cC$, we define the mapping $\varphi_\cC\colon V\to \{1,2,\ldots,N\}$ by $\varphi_\cC(P)=i$ if and only if $P\in C_i$. We call the numbers $1,\ldots, N$ colors and the sets $C_1,\ldots C_N$ color classes. We call a hyperedge $H\in E$ \emph{rainbow (with respect to $\cC$)} if no two points of $H$ have the same color; that is, $|H\cap C_i|\leq 1$ for all $1\leq i\leq N$. The upper chromatic number (or shortly UCN) of the hypergraph $\cH$, denoted by $\UCN(\cH)$, is the maximal number $N$ for which $\cH$ admits a strict $N$-coloring without rainbow hyperedges. Let us call such a coloring \emph{proper} or \emph{rainbow-free}. It is easy to see that for an ordinary graph $G$ (that is, a $2$-uniform hypergraph), $\UCN(G)$ is just the number of connected components of $G$.

As one can see, the above defined hypergraph coloring problem is a counterpart of the traditional one, where we seek the least number of colors with which we can color the vertices of a hypergraph while forbidding hyperedges to contain two vertices of the same color. The general mixed hypergraph model, introduced by Voloshin \cite{V1,V2}, combines the above two concepts. This mixed model is better known but here we do not discuss it; the interested reader is referred to \cite{Iljics}.

It is clear that if we find a vertex set $T\subset V$ in $\cH$ which intersects every hyperedge in at least two points, then by coloring the points of $T$ with one color and all the other points of $V$ by mutually distinct colors, we obtain a proper, strict $(|V|-|T|+1)$-coloring.

\begin{definition}\label{transversal}
Let $\cH=(V;E)$ be a hypergraph, $t$ a nonnegative integer. A vertex set $T\subset V$ is called a \emph{$t$-transversal of $\cH$} if $|T\cap H|\geq t$ for all $H\in E$. The size of the smallest $t$-transversal of $\cH$ is denoted by $\tau_t(\cH)$.
\end{definition}

\begin{definition}\label{trivcolor}
We say that a coloring of $\cH$ is \emph{trivial} if it contains a monochromatic $2$-transversal.
\end{definition}

As seen above, the best trivial colorings immediately yield a lower bound for $\UCN(\cH)$.

\begin{proposition}\label{trivbound}
For any hypergraph $\cH$,
\[\UCN(\cH)\geq |V|-\tau_2(\cH)+1.\]
\end{proposition}

Two general problems are to determine whether this bound is sharp (for a particular class of hypergraphs) and to describe the colorings attaining the upper chromatic number. In this paper, the hypergraphs we consider consist of the points of the $n$-dimensional projective space $\PG(n,q)$ over the finite field $\GF(q)$ of $q$ elements with its $k$-dimensional subspaces as hyperedges, $n\geq 2$, $1\leq k\leq n-1$. We denote this hypergraph by $\cH(n,k,q)$; however, we usually take into account the richer structure of $\PG(n,q)$ when working in $\cH(n,k,q)$. The study of this particular case was started in the mid-nineties by Bacs\'o and Tuza \cite{TB}, who established general bounds for the upper chromatic number of arbitrary finite projective planes (considered as a hypergraph whose points and hyperedges are the points and lines of the plane). Let us introduce the notation $\theta_{q,k}=\theta_k=q^k+q^{k-1}+\ldots+q+1=\frac{q^{k+1}-1}{q-1}$ for the number of points in a $k$-dimensional projective space of order $q$. We recall that a projective plane of order $q$ has $\theta_2=q^2+q+1$ points.

\begin{result}[Bacs\'o, Tuza \cite{TB}]
Let $\Pi_q$ be an arbitrary finite projective plane of order $q$, and let $\tau_2(\Pi_q)=2(q+1)+c(\Pi_q)$. Then
\[\UCN(\Pi_q)\leq q^2-q-\frac{c(\Pi_q)}{2}+o(\sqrt{q}).\]
\end{result}

Note that Proposition \ref{trivbound} claims $\UCN(\Pi_q)\geq q^2-q-c(\Pi_q)$. Recently, Bacs\'o, H\'eger, and Sz\H{o}nyi have obtained exact results for the Desarguesian projective plane $\PG(2,q)$.

\begin{result}[Bacs\'o, H\'eger, Sz\H{o}nyi \cite{BHSz}]\label{BHSz}
Let $q=p^h$, $p$ prime. Suppose that either $q>256$ is a square, or $p\geq 29$ and $h\geq3$ odd. Then $\UCN(\PG(2,q))=\theta_2-\tau_2(\PG(2,q))+1$, and equality is reached only by trivial colorings.
\end{result}

In this work, we determine $\UCN(\cH(n,k,q))$ for many parameters and aim not only to characterise trivial colorings as the only ones achieving the upper chromatic number of the hypergraph $\cH(n,k,q)$, but to obtain results showing that proper colorings of $\cH(n,k,q)$ using a little less number of colors than $\UCN(\cH(n,k,q))$ are trivial; in other words, to prove that trivial colorings are stable regarding the number of colors. For the sake of convenience, we will formulate our results in three theorems for the hypergraph $\cH(n,n-k,q)$. Let us note that if $k<\frac{n}{2}$, then $\tau_2(\cH(n,n-k,q))=2\theta_k$, where equality can be reached by the union of two disjoint $k$-spaces, but not much is known if $k\geq \frac{n}{2}$; some details are given in Section \ref{sec:blsets}.

\begin{theorem}\label{ucnthm}
Let $n\geq 3$, $1\leq k<\frac{n}{2}$, and assume that $q\geq 17$ if $k=1$ and $q\geq 13$ if $k\geq 2$. Then
\[
\UCN(\cH(n,n-k,q))=\theta_n-\tau_2(\cH(n,n-k,q))+1=\theta_n-2\theta_k + 1.
\]
\end{theorem}

\begin{theorem}\label{ucnstabp}
Let $n\geq 2$, $q=p^h$, $p$ prime, $1\leq k\leq n-1$. Suppose that
\begin{itemize}
\item  $\delta = \frac12\left( (\sqrt{2}-1)q^k - 3\theta_{k-1} - 8 \right)\geq0$ and $q\geq 11$ if $h=1$,
\item  $\delta = \frac12\left(q^{k-1}-\theta_{k-2}-3\right)$, $k\geq 2$ and $q\geq 25$ if $h\geq2$.
\end{itemize}
Under these assumptions the following hold:
\begin{compactenum}
\item[a)] If $k < \frac{n}{2}$, then any rainbow-free coloring of $\cH(n,n-k,q)$ using
\[
N\geq \theta_n-\tau_2(\cH(n,n-k,q))+1-\delta = \theta_n-2\theta_k + 1 - \delta
\]
colors contains a monochromatic pair of disjoint $k$-spaces, and hence is trivial.
\item[b)] If $k \geq \frac{n}{2}$, then
\[
\UCN(\cH(n,n-k,q)) < \theta_n-2\theta_k + 1 - \delta.
\]
\end{compactenum}
\end{theorem}

Note that the stability gap $\delta$ in the above result is far much weaker in the non-prime case (in particular, the case $k=1$ is missing). The next theorem gives a much better result at the expense of requiring much stronger assumptions on the order and the characteristic of the field.

\begin{theorem}\label{ucnstabq}
Let $n\geq 2$, $q=p^h$, $p$ prime, $h\geq 2$, $1\leq k\leq n-1$. Suppose that $p\ge11$, $q^k\ge 239$ and $\delta=\frac{q^k}{200} - \theta_{k-1}-\frac32$. Then any rainbow-free coloring of $\cH(n,n-k,q)$ using
\[
N\geq \theta_n-2\theta_k + 1 - \delta
\]
colors contains a monochromatic $2$-transversal, and hence is trivial.
\end{theorem}

The requirements on $q$ and $N$ in the above theorem could be chosen differently, see Remark \ref{ucnthm:remark} for the details. Note that Theorem \ref{ucnstabq} is not phrased in terms of $\tau_2(\cH(n,n-k,q))$, the parameter found in the trivial lower bound Proposition \ref{trivbound}. As noted earlier, $\tau_2(\cH(n,n-k,q))=2\theta_k$ if $k<\frac{n}{2}$. If $k=\frac{n}{2}$, then \cite[Corollary 4.13]{DBHSzVdV} asserts the existence of a $2$-fold $k$-blocking set in $\PG(2k,q)$ (in finite geometrical language, $t$-transversals of $\cH(n,n-k,q)$ are called $t$-fold $k$-blocking sets, see Section \ref{sec:blsets}) of size $2q^k + 2\frac{q^k-1}{p-1}$, where $q=p^h$, $p>5$ prime, $h\geq2$. Thus, if $p \geq 409$, then $\tau_2(\cH(2k,k,q))\leq 2\theta_k + \delta$, whence Theorem \ref{ucnstabq} yields that the trivial bound is again sharp for $\cH(2k,k,q)$, regardless the exact value of $\tau_2(\cH(2k,k,q))$.

In the proof of the above theorem, we rely on weighted $2$-fold blocking sets as well, so we devote the next section to this topic, and we obtain the following new result which, in fact, follows from the similar Theorem \ref{tmodpsetthm} about $t \pmod p$ sets. The precise definitions are given in the next section.

\begin{theorem}\label{blsetthm}
Let $\cB$ be a minimal weighted $t$-fold $k$-blocking set of $\PG(n,p)$, $p$ prime. Assume that $|\cB| \le (t+\frac{1}{2})p^{k}-\frac12$ and $t\leq \frac38p+1$. Then $\cB$ is the (weighted) union of $t$ not necessarily distinct $k$-dimensional subspaces.
\end{theorem}

\section{Small, weighted multiple $(n-k)$-blocking sets}\label{sec:blsets}

For the sake of convenience, we will refer to $(n-k)$-blocking sets instead of $k$-blocking sets throughout this section.

\subsection{Preliminary notation and results}

\begin{definition}
An \emph{$m$-space} is a subspace of $\PG(n,q)$ of dimension $m$ (in projective sense). A point-set $\cB$ in $\PG(n,q)$ is called a \emph{$t$-fold $(n-k)$-blocking set} if every $k$-space intersects $\cB$ in at least $t$ points. A point $P$ of $\cB$ is \emph{essential} if $\cB\setminus\{P\}$ is not a $t$-fold $(n-k)$-blocking set; in other words, if there is a $k$-space through $P$ that intersects $\cB$ in precisely $t$ points. $\cB$ is called \emph{minimal}, if all of its points are essential; in other words, if $\cB$ does not contain a smaller $t$-fold $(n-k)$-blocking set.
\end{definition}

In $\PG(n,q)$, every $(n-k)$-space intersects every $k$-space non-trivially. If $n-k<\frac{n}{2}$, it is easy to find two (or more, say, $t$) disjoint $(n-k)$-spaces, whose union is clearly a $2$-fold (or $t$-fold) $(n-k)$-blocking set of size $2\theta_{n-k}$. If $n-k\geq \frac{n}{2}$, this does not work and, in fact, not much is known even about the size of a smallest double $(n-k)$-blocking set, let alone its structure. Even for the particular case $n=2k$, the first and, so far, only general construction for small double $(n-k)$-blocking sets appeared in \cite{DBHSzVdV}. Note that, however, \emph{weighted} $t$-fold blocking sets can be obtained easily in this way.

\begin{definition}
A weighted point set of $\PG(n,q)$ is a multiset $\cB$ of the points of $\PG(n,q)$. We may refer to the multiplicities of the points of $\cB$ via a function $w=w_\cB$ mapping the point set of $\PG(n,q)$ to the set of non-negative integers, where $w$ is also called a weight function; points not contained in $\cB$ have weight zero by $w$ and, vice versa, zero weight points are considered to be not in $\cB$. We call a weighted point set $\cB$ of $\PG(n,q)$ a weighted $t$-fold $(n-k)$-blocking set if for every $k$-space $U$, $\sum_{P\in U}w(P)\geq t$, and $\cB$ is called minimal if decreasing the weight of any point results in a $k$-space violating the previous property; in other words, if $\cB$ does not contain a strictly smaller $t$-fold $(n-k)$-blocking set, where the size of a weighted point set is defined as the sum of weights in it. Also, for any point set $S$, $|S\cap\cB|$ is defined as $\sum_{P\in S}w(P)$, and in general, any quantity referring to a number of points of $\cB$ is usually considered with multiplicities. E.g., an $i$-secant line $\ell$ (with respect to $\cB$) is a line such that $|\ell\cap\cB|=i$.
\end{definition}

We also refer to $1$-fold and $2$-fold blocking sets as blocking sets and double blocking sets, respectively; the term multiple blocking set refers to a $t$-fold blocking set with $t\geq 2$. We call a point of weight one simple. It is easy to see that a weighted $t$-fold $k$-blocking set must contain at least $t\theta_k$ points unless $t\geq q+1$. We include this supposedly folklore result with proof for the sake of completeness.
\begin{proposition}\label{folklore}
Let $\cB$ be a $t$-fold $(n-k)$-blocking set in $\PG(n,q)$. If $t\leq q$, then $|\cB|\geq t\theta_{n-k}$.
\end{proposition}
\begin{proof}
We prove by induction on $k$. If $k=1$, we may take a point $P\notin\cB$ (otherwise $|\cB|\geq \theta_n > q\theta_{n-1}$ and there is nothing to prove). There are $\theta_{n-1}$ lines through $P$, each containing at least $t$ points of $\cB$, whence $|\cB|\geq t\theta_{n-1}$. Suppose now $k\geq 2$. If $\cB$ is an $(n-k+1)$-blocking set then, by induction, $|\cB|\geq \theta_{n-k+1}=q\theta_{n-k}+1 > t\theta_{n-k}$ and we are done. If there is a $(k-1)$-space $\Pi$ disjoint from $\cB$, then each of the $\theta_{n-k}$ distinct $k$-spaces containing $\Pi$ intersects $\cB$ in at least $t$ points, whence $|\cB|\geq t\theta_{n-k}$.
\end{proof}

Note that $t\leq q$ is necessary here, as if $\cB$ contains each point of an $(n-k+1)$-space with weight one, then $\cB$ is a $(q+1)$-fold $(n-k)$-blocking set of size $\theta_{n-k+1}=q\theta_{n-k} + 1 < (q+1)\theta_{n-k}$; moreover, adding $s$ further $(n-k)$-spaces to $\cB$ we obtain a weighted $(q+1+s)$-fold $(n-k)$-blocking set of size less than $(q+1+s)\theta_{n-k}$ for any $s\geq 0$.

A stability result for weighted $t$-fold $(n-k)$-blocking sets of size close to this lower bound was proven by Klein and Metsch \cite[Theorem 11]{KM}.

\begin{result}[Klein, Metsch \cite{KM}]\label{KM}
Let $\cB$ be a weighted $t$-fold $(n-k)$-blocking set in $\PG(n,q)$. Suppose that $|\cB|\leq t\theta_{n-k} + r\theta_{n-k-2}$, where $t$ and $r$ satisfy the following:
\begin{enumerate}
\item[a)] {$1\leq t\leq \frac{q+1}{2}$;}
\item[b)] $t+r\leq q$, $r\geq 0$ is an integer;
\item[c)] any blocking set of $\PG(2,q)$ of size at most $q+t$ contains a line.
\end{enumerate}
Then $\cB$ contains the (weighted) union of $t$ not necessarily distinct $(n-k)$-spaces.
\end{result}

Let us remark that for $k=1$ (that is, when $\cB$ is a $t$-fold weighted blocking set with respect to lines), \cite[Theorem 7]{KM} shows that condition \emph{c)} can be omitted in the above result. However, a blocking set of $\PG(2,q)$ not containing a line must contain at least $q+\sqrt{q}+1$ points in general (see \cite{Bruen} by Bruen), and, according to the following result of Blokhuis, at least $\frac32(q+1)$ if $q$ is prime, hence condition \emph{c)} holds accordingly.

\begin{result}[Blokhuis \cite{Husi}]\label{Husi}
Suppose that $\cB$ is a blocking set in $\PG(2,p)$, $p$ prime, not containing a line. Then $|\cB|\geq \frac32(p+1)$.
\end{result}

The following two theorems will be very useful for us.

\begin{result}[Harrach \cite{Harrach}]\label{Harrach}
Suppose that a weighted $t$-fold $k$-blocking set $\cB$ in $\PG(n,q)$ has less than $(t+1)q^k+\theta_{k-1}$ points. Then $\cB$ contains a unique minimal weighted $t$-fold $k$-blocking set $\cB'$.
\end{result}

A theorem of the below type is often called a $t \pmod p$ result.

\begin{result}[Ferret, Storme, Sziklai, Weiner \cite{FSSzW}]\label{FSSzW}
Let $\cB$ be a minimal weighted $t$-fold $(n-k)$-blocking set of $\PG(n,q)$, $q=p^h$, $p$ prime, $h\geq1$, of size $|\cB|=tq^{n-k}+t+k'$, with $t+k'\leq\frac{q^{n-k}-1}{2}$. Then $\cB$ intersects every $k$-space in $t \pmod p$ points \cite[Theorem 4.2]{FSSzW}. Moreover if $e\ge1$ denotes the largest integer for which each $k$-space intersects $\cB$ in $t \pmod{p^e}$ points, then $|\cB| > tq^{n-k} + \frac{q^{n-k}}{p^e+1} - 1$ \cite[Corollary 5.2]{FSSzW}.
\end{result}

Finally, we recall that the number of $(k+1)$-spaces containing a fixed $k$-space in $\PG(n,q)$ is $\theta_{n-k-1}$. This can be seen easily by taking an $(n-k-1)$-space disjoint from the fixed $k$-space and observing that each appropriate $(k+1)$-space intersects it in a unique point.

\subsection{Proof of Theorem \ref{blsetthm}}

We prove a theorem closely related to Theorem \ref{blsetthm} by considering an analogous problem in a slightly more general setting.

\begin{definition}\label{tmodpset}
Let us call a weighted point set $\cB$ in $\PG(n,q)$ \emph{a $t \pmod p$ set} with respect to the $k$-dimensional subspaces if $\cB$ intersects every $k$-space ($1\leq k\leq n-1$) in $t \pmod p$ points (counted with weights).
\end{definition}

Clearly, $t \pmod p$ sets are $t$-fold blocking sets if $t<p$ and, by Result \ref{FSSzW}, small minimal $t$-fold blocking sets are $t \pmod p$ sets.

\begin{theorem}\label{tmodpsetthm}
Let $\cB$ be a $t \pmod p$ set with respect to the $k$-dimensional subspaces in $\PG(n,p)$, $p$ prime. Suppose that $t\leq \frac38p + 1$ and $|\cB|\leq (t+1)\theta_{n-k} + p - 2$. Then $|\cB|= t\theta_{n-k}$ and $\cB$ consists of the weighted union of $t$ not necessarily distinct $(n-k)$-spaces.
\end{theorem}
\begin{proof}
The proof will use induction on $k$. Clearly, $\cB$ is a $t$-fold $(n-k)$-blocking set. We will need the existence of a point not in $\cB$. This follows if $|\cB|<\theta_n$. If $p-1\geq t+1$, then our assumption gives $|\cB|\leq (p-1)\theta_{n-1}+p-2 = \theta_{n} - 1 - \theta_{n-1} + p - 2 < \theta_n$. If $p\leq t+1$, then from $t\leq \frac38p+1$ it follows that $p\leq 3$ must hold. If $p=2$, then $t=1$ and $|\cB|\leq (t+1)\theta_{n-1} + p - 2 = \theta_n - 1$. If $p=3$, the problematic case is $t=2$, when $|\cB|\leq 3\theta_{n-1} + 1 = \theta_n$. If $|\cB|=\theta_n$ and $\cB$ contains every point of the space, then it is clearly a $1 \pmod p$ set for every subspace, in contradiction with $t=2$. Hence we always find a point not contained in $\cB$.

\noindent\textbf{Case 1: $k=1$ (and $n\geq 2$).}
Notice first that every point of $\cB$ has weight at most $t$. Indeed, by taking the weights of the points modulo $p$, we may assume that no point has weight at least $p$; and if $t+1\leq w(P)\leq p-1$ for a point $P$, then all the $\theta_{n-1}$ lines through $P$ must contain at least $p+t-w(P)$ more weights, whence $|\cB|\geq w(P) + (p+t-w(P))\theta_{n-1}\geq p - 1 + (t+1)\theta_{n-1}$, a contradiction.

It follows from Results \ref{KM} and \ref{Husi} that the assertion holds if $|\cB|=t\theta_{n-1}$, hence we may assume that $|\cB|> t\theta_{n-1}$ and prove by contradiction.

We will call lines that are neither $t$-secants (to $\cB$), nor contained fully in $\cB$ \emph{long} lines; lines contained in $\cB$ will be referred to as \emph{full} lines. Non-$t$-secant lines are, therefore, either full or long. Long lines exist as on any point not in $\cB$ (an \emph{outer} point) we find a line intersecting $\cB$ in more than $t$ points, since $|\cB|=t\theta_{n-1}$ would follow otherwise. Suppose that the minimum weight of a long line is $sp+t$. Clearly, $1\leq s \leq t-1$ (the weight of a long line is at most $tp$). Let $\ell$ be a long line of weight $sp+t$, and let $P\in\ell\setminus\cB$. We want to show that for any $2$-space $\Pi$ containing $\ell$, there is a long line through $P$ in $\Pi$ different from $\ell$. Fix such a plane $\Pi$ (if $n=2$, then this is unique) and suppose to the contrary. Let $\cB'=\cB\cap\Pi$. Then, looking around from $P$ in $\Pi$, $|\cB'|= (p+1)t + sp$. Similarly as before, there must be a non-$t$-secant line on any point $R\in\Pi$; in other words, long and full lines form a blocking set in the dual plane of $\Pi$. It follows that long lines cover each outer point of $\Pi$ exactly once. Moreover, the number of non-$t$-secant lines must be at least $\frac32(p+1)$ for the following reason. By Blokhuis' Result \ref{Husi}, a blocking set of $\PG(2,p)$ of size less than $\frac32(p+1)$ contains a line. In our setting this situation would result in a point $Q$ through which all lines are either long or full. But then $(2t-1)p + t\geq tp+t+sp = |\cB'|\geq w(Q) + (p+1)(p+t-w(Q)) \geq t+(p+1)p$, a contradiction even under $t < \frac12p+1$.

Let $e$ be a $t$-secant to $\cB'$ (such a line exists as seen above). Let $P_1,\ldots,P_r$ be the mutually distinct points of $e\cap\cB'$, $1\leq r \leq t$. Let $h_1(P_i)$ and $h_2(P_i)$ denote the number of full and long lines on $P_i$, respectively, and let $h_1$ and $h_2$ be the total number of full and long lines, respectively; then $h_1+h_2 \geq \frac32(p+1)$. Looking around from $P_i$ we see that
\[
(p+1)t + sp = |\cB'|\geq w(P_i) + (p+1)(t-w(P_i)) + h_1(P_i)p + h_2(P_i)sp,
\]
whence $w(P_i)+s\geq h_1(P_i)+sh_2(P_i)$. Let $h_2':=h_2-(p+1-r)$ be the number of long lines intersecting $e$ in a point of $\cB'$. Then $h_1+h_2'\geq \frac32(p+1) - (p+1-r) \geq \frac12(p+1)+r$, and we obtain that
\[
t+rs = \sum_{i=1}^r (w(P_i) + s) \geq  \sum_{i=1}^r (h_1(P_i) + sh_2(P_i)) = h_1 + sh_2' =  h_1+s(h_2-(p+1-r)),
\]
whence
\begin{eqnarray*}
t &\!\geq\!& h_1 + s(h_2 -(p+1)) > h_1 + s\left(\left(\frac{3(p+1)}{2}-h_1\right) - (p+1)\right)
= h_1 + s\left(\frac{p+1}{2}-h_1\right)\\
&\!=\!& (s-1)\left(\frac{p+1}{2} -h_1 \right) +\frac{p+1}{2}.
\end{eqnarray*}
As $t < \frac12(p+1)$, it follows that $s\geq 2$ and $h_1 > \frac12(p+1)$.

It is clear that there are at least $p+1-r\geq 2$ long lines. Take now two long lines and the $h_1\geq \frac12p+1$ full lines one by one. The first line contains at least $sp+t$ weights of $\cB'$. The second line may intersect it in a point of weight at most $t$, hence we see at least $sp$ more weights on it. Turning to the full lines, the $i$th full line contains at least $p+1 - 2 - (i-1) = p - i$ points of $\cB'$ not contained by any of the previous lines. Altogether we obtain
\begin{eqnarray*}
(p+1)t + sp = |\cB'| &\geq& 2sp + t  + \sum_{i=1}^{\frac{p}{2}+1}(p-i) = 2sp + t + \left(\frac{p}{2}+1\right)p - \binom{\frac{p}{2}+2}{2} \\
 &=& 2sp + t + \frac{p^2}{2} + p - \frac{p^2}{8}-\frac{3p}{4} -1,
\end{eqnarray*}
whence
\[
t\geq \frac38p + s + \frac14 -\frac1p > \frac38p + 1,
\]
a contradiction. Thus we see that all planes containing $\ell$ indeed contain at least one other long line through $P$, so we find at least $1+\theta_{n-2}$ long lines through $P$, hence on all the $\theta_{n-1}$ lines through $P$ we find that $|\cB|\geq t\theta_{n-1} + (\theta_{n-2}+1)sp = t\theta_{n-1} + s(\theta_{n-1}-1) + sp \geq (t+1)\theta_{n-1} + p -1$, a contradiction.

\textbf{Case 2: $2\leq k \leq n-1$ (and $n\geq 3$).}
Take a point $P \notin \cB$ in $\PG(n,p)$. Project the points of $\cB$ from $P$ into an arbitrary hyperplane $H$. We get a weighted point set $\tb \subseteq \Pi$ for which $|\tb| = |\cB|$. Let $W$ be a $(k-1)$-space in $H$, and let $U = \langle P, W\rangle$ be the $k$-space spanned by $P$ and $W$. Then $|W\cap \tb| = |U \cap \cB|$, hence $\tb$ is a $t \pmod p$ set with respect to $(k-1)$-spaces in the $(n-1)$-space $H$ thus, by induction on $k$, $|\cB| = |\tb| = t\theta_{n-1 - (k-1)}=t\theta_{n-k}$. Results \ref{KM} and \ref{Husi} finish the proof.
\end{proof}

Theorem \ref{blsetthm} now follows from Theorem \ref{tmodpsetthm} and the $t\pmod p$ Result \ref{FSSzW}.

\section{On the upper chromatic number of $\cH(n,n-k,q)$}

\subsection{Proof of Theorems \ref{ucnthm} and \ref{ucnstabp}}

The steps of the proof have a lot in common with those in \cite{BHSz}. We recall that we want to color the points of $\PG(n,q)$ with as many colors as possible so that each $(n-k)$-space contains two equicolored points. For two points $P$ and $Q$, $PQ$ denotes the line joining them.

\begin{definition}
Let $[m]_q=\frac{q^{m}-1}{q-1}=q^{m-1}+q^{m-2}+\ldots+q+1$. Let $[m]_q!=\prod_{i=1}^m[i]_q$, where $[0]_q!=1$, and let $\qbinom{n}{m}=\frac{[n]_q!}{[m]_q![n-m]_q!}=\frac{(q^n-1)(q^{n-1}-1)\ldots(q^{n-m+1}-1)}{(q^m-1)(q^{m-1}-1)\ldots(q-1)}$ denote the number of $m$ dimensional subspaces of an $n$ dimensional \emph{vector space}; thus the number of $m$-spaces in $\PG(n,q)$ is $\qbinom{n+1}{m+1}$.
\end{definition}

Note that $\qbinom{k+1}{1}=\theta_k$. Let us collect some facts regarding the above defined $q$-binomial coefficients.

\begin{lemma}\label{lemmata}
Let $q\geq2$, $n\geq1$, $s\geq 0$. Then
\begin{itemize}
\item[a)]{$\qbinom{n}{1}=q^{n-1}+\qbinom{n-1}{1}$, that is, $\theta_n=q^{n-1}+\theta_{n-1}$;}
\item[b)]{the number of $m$-spaces containing a given $k$-space in $\GF(q)^n$ is $\qbinom{n-k}{m-k}$;}
\item[c)]{$\theta_s=\qbinom{s+1}{1}<\left (1+\frac{1}{q-1} \right )q^s\leq \left (1+\frac{2}{q} \right )q^s$;}
\item[d)]{if $s\leq n-1$, then $\frac{\qbinom{n+1}{n-s+1}}{\qbinom{n-1}{s}}\geq q^{2s}$.}
\end{itemize}
\end{lemma}
\begin{proof}
The first statement is trivial. As for the second one, let $U$ be the given $k$-space. The quantity in question is just the number of $(m-k)$-spaces in the $(n-k)$-dimensional quotient space $\GF(q)^{n}/U$. The third assertion is trivial for $s=0$; otherwise $\theta_s=q^s+q^{s-1}+\frac{q^{s-1}-1}{q-1}<q^s+(1+\frac{1}{q-1})q^{s-1}$. Finally, regarding the fourth: it is trivial if $s=0$; for $s\geq 1$,
\[\frac{\qbinom{n+1}{n-s+1}}{\qbinom{n-1}{s}}=\frac{[n+1]_q![s]_q![n-1-s]_q!}{[n-s+1]_q![s]_q![n-1]_q!}=\frac{(q^{n+1}-1)(q^n-1)}{(q^{n-s+1}-1)(q^{n-s}-1)}\]
\[>\frac{q^{2n+1}-2q^{n+1}}{q^{2n-2s+1}-q^{n-s+1}}=\frac{q^{n+s}-2q^{-s}}{q^{n-s}-1}> \frac{q^{n+s}-2q^{-s}+1}{q^{n-s}}\]
\[=q^{2s}-2q^{-n}+q^{s}q^{-n}\geq q^{2s},\]
as $s\geq 1$ and $q\geq2$.
\end{proof}

\textbf{General notation and assumptions.} Suppose that a strict proper coloring $\cC$ of $\cH(n,n-k,q)$ using $N$ colors is given. We denote the color classes of $\cC$ by $C_1,\ldots,C_N$. For the sake of simplicity, we will compare $N$ with  $\theta_n-2\theta_k+1$, hence we define the \emph{deficit} $d=d(\cC)$ of $\cC$ by $N= \theta_n-2\theta_k+1-d$ which, in principle, may be negative as well. Note that the number $\delta$ in Theorems \ref{ucnstabp} and \ref{ucnstabq} is an upper bound on $d$. Without loss of generality we may assume that $C_1,\ldots,C_m$ are precisely the color classes of size at least two for some $m\geq1$. Let $\cB=\cB(\cC)=\cup_{i=1}^{m}C_i$.

\begin{definition}
We say that a color class $C$ \emph{colors the $(n-k)$-space $U$} if $|C\cap U|\geq 2$.
\end{definition}

As every $(n-k)$-space must be colored by at least one of the color classes among $C_1,\ldots,C_m$, we clearly see that $\cB$ is a $2$-fold $k$-blocking set.

\begin{prop}\label{trivprop}\mbox{}
\vspace{-3mm}
\begin{itemize}
\item[a)]{$m= |\cB|-2\theta_k-d+1$}
\item[b)]{$m\leq 2\theta_k+d-1$}
\item[c)]{$|\cB|\leq 4\theta_k+2(d-1)$}
\end{itemize}
\end{prop}
\begin{proof}
The first assertion follows from $\theta_n-|\cB|+m=N = \theta_n-2\theta_k+1-d$. As $|C_i|\geq2$ for all $1\leq i\leq m$, $|\cB|\geq 2m$. This and the previous equality imply $m\leq 2\theta_k+d-1$ and $|\cB|\leq 4\theta_k+2(d-1)$.
\end{proof}

\begin{lemma}\label{maxcolored}
A color class $C$ colors at most $\binom{|C|}{2}\qbinom{n-1}{k}$ distinct $(n-k)$-spaces.
\end{lemma}
\begin{proof}
If $C$ colors an $(n-k)$-space $U$, then $U$ contains a line spanned by the points of $C$. The number of such lines is at most $\binom{|C|}{2}$. By Lemma \ref{lemmata}, the number of $(n-k)$ spaces containing a given line is $\qbinom{n-1}{n-k-1}=\qbinom{n-1}{k}$.
\end{proof}

The next proposition says that $\cB$ cannot be too large; roughly speaking, $|\cB|\leq (4-\sqrt{2})q^k+2d+o(q^k)$.

\begin{prop}\label{big}
Suppose that $d\leq \alpha q^k$, and $q > \left(\frac{5}{\sqrt{2}}-2-\alpha-\frac{4}{q}\right)^{-1}>0$. Then $|\cB|<(4-\sqrt{2})q^k+4\theta_{k-1}+2d+2$.
\end{prop}
\begin{proof}
As every $(n-k)$-space must be colored, by Lemma \ref{maxcolored} and convexity we have
\[\dqbinom{n+1}{n-k+1}\leq \sum_{i=1}^{m}\binom{|C_i|}{2}\dqbinom{n-1}{k}\leq \left(\binom{|\cB|-2(m-1)}{2}+(m-1)\right)\dqbinom{n-1}{k}.\]
By Lemma \ref{lemmata} $d)$ and Proposition \ref{trivprop} $a)$, $b)$
\[q^{2k}\leq \binom{|\cB|-2(m-1)}{2}+(m-1)\leq \binom{4\theta_k-|\cB|+2d}{2}+2\theta_k+d-2.\]
Suppose to the contrary that $|\cB|\geq 4\theta_k-\sqrt{2}q^k+2d+2=(4-\sqrt{2})q^k+4\theta_{k-1}+2d+2$ (here we use Lemma \ref{lemmata} $a)$). Then by the assumption and Lemma \ref{lemmata} $c)$, the right-hand-side of the above expression is at most
\[\binom{\sqrt{2}q^k-2}{2}+2\theta_k+d-2 < q^{2k}-\left(\frac{5}{\sqrt{2}} - 2\left(1+\frac{2}{q}\right) - \alpha\right) q^k+1<q^{2k},\]
a contradiction.
\end{proof}

The following lemma will be very useful as it provides us large color classes if $\cB$ is not large. The proof is based on Result \ref{Harrach}. Right now, we do not need the following stronger version of this lemma since our blocking set has no weights, but respecting its future use we will state it in a more general setting. This version can deal with colorings which come from weighted blocking sets.

\begin{lemma}\label{esspt}
Suppose that a color class $C$ contains a simple essential point $P$ of $\cB$. Then there exists a set of simple points $S\subset C\setminus\{P\}$ such that $|S|\geq3q^k+\theta_{k-1}-|\cB|$, and for any point $Q\in S$ there exists an $(n-k)$ space $U$ such that $U\cap\cB=\{P;Q\}$ (so these points are essential for $\cB$). In particular, $|C|\geq 3q^k+\theta_{k-1}+1-|\cB|$.
\end{lemma}
\begin{proof}
For $Q\in\cB$, $Q\neq P$, let $P\sim Q$ iff $Q$ is also a simple point and there exists an $(n-k)$-space $U$ such that $U\cap\cB=\{P,Q\}$. As $P$ is simple and essential, we find at least one such point. Let $\{Q_1,\ldots,Q_r\}=\{Q\in\cB\colon P\sim Q\}$. For all $1\leq i\leq r$, take a point $R_i$ from $PQ_i\setminus\{P,Q_i\}$, and let $R=\{R_1,\ldots, R_r\}$. Then the set $(\cB\cup R)\setminus\{P\}$ is also a $2$-fold $k$-blocking set. Thus $\cB\cup R$ contains two different minimal $2$-fold $k$-blocking sets, so by Harrach's Result \ref{Harrach} we have $|\cB|+r\geq 3q^k+\theta_{k-1}$. As $Q_1,\ldots,Q_r$ must have the same color as $P$, the proof is finished.
\end{proof}

Now we are ready to show that if $\cB$ is not large, then it is, in fact, quite small. Roughly speaking, if $|\cB|<3q^k$, then $|\cB|<2q^k+2d+o(q^k)$. We will use that $3q^k+\theta_{k-1}-2\theta_{k}=3q^k+\theta_{k-1}-2\left(q^k+\theta_{k-1}\right)=q^k-\theta_{k-1}$.

\begin{prop}\label{medium}
Let $\beta\geq 2$. Assume $|\cB|\leq 3q^k+\theta_{k-1}+1-\beta$ and $(\beta-4)q^k>(\beta+4)\theta_{k-1}+\beta(2d+\beta-3)$. Then $|\cB|<2\theta_{k}+2d-2+\beta$.

\end{prop}
\begin{proof}
By Result \ref{Harrach} there is a unique minimal $2$-fold $k$-blocking set $\cB'$ contained in $\cB$. By Lemma \ref{esspt} we know that if a color class contains a point of $\cB'$, then it contains at least $3q^k+\theta_{k-1}+1-|\cB|$ points of it, while all other color classes in $\cB$ have at least two points. This and Proposition \ref{trivprop} $a)$ imply that
\[ |\cB|-2\theta_{k}+1-d= m\leq \frac{|\cB'|}{3q^k+\theta_{k-1}+1-|\cB|}+\frac{|\cB|-|\cB'|}{2}\leq\]\[\leq \frac{2\theta_{k}}{3q^k+\theta_{k-1}+1-|\cB|}+\frac{|\cB|-2\theta_{k}}{2},\]
so
\[\left(|\cB|-\left(2\theta_{k}+2d-2\right)\right)\left(3q^k+\theta_{k-1}+1-|\cB|\right)\leq 4\theta_{k}.\]
The left-hand side expression is concave in $|\cB|$. Substituting either $|\cB|=2\theta_{k}+2d-2+\beta=2q^k+2\theta_{k-1}+2d-2+\beta$ or $|\cB|= 3q^k+\theta_{k-1}+1-\beta$ we obtain
\[\beta \left(3q^k+\theta_{k-1}-2\theta_{k}-2d+3-\beta\right)\leq 4\theta_{k},\]
which, due to simple calculations and rearrangement, leads to $(\beta-4)q^k\leq (\beta+4)\theta_{k-1}+\beta(2d+\beta-3)$, a contradiction. As $|\cB|\leq 3q^k+\theta_{k-1}+1-\beta$, we conclude that $|\cB|<2\theta_{k}+2d-2+\beta$ must hold.
\end{proof}

Using these results, we may assume that $\cB$ is quite small. As shown by the next proposition, this immediately gives the desired result on the upper chromatic number provided that $\cB$ contains the union of two disjoint one-fold blocking sets, which property can be deduced from a stability type result on multiple blocking sets like Theorem \ref{blsetthm} or Result \ref{KM}; however, the strength of the result obtained in this way will be utterly dependent on the strength of the stability result.

\begin{prop}\label{small}
Suppose that $\cB$ contains two disjoint $k$-blocking sets, $U_1$ and $U_2$. If the coloring is nontrivial, then $|U_1|+|U_2|\geq 4 (3q^k-|\cB|+\theta_{k-1})$; in particular, $|\cB|\geq 2.4q^k+0.8\theta_{k-1}$ and, if $U_1$ and $U_2$ are $k$-spaces, then $|\cB|\geq 2.5q^k+\frac12\theta_{k-1}$.
\end{prop}
\begin{proof}
We may assume that $|\cB|<3q^k+\theta_{k-1}$, otherwise the assertions are trivial. Then, by Result \ref{Harrach}, $\cB'=U_1\cup U_2$ is precisely the set of essential points of $\cB$. If the coloring is not trivial, then there are at least two colors used in $\cB'$, say, red and green. Without loss of generality we may take a red point $P\in U_1$. By Lemma \ref{esspt}, we find a set $S$ of essential points of $\cB$ such that $|S|=3q^k-|\cB|+\theta_{k-1}$, and for each point $Q\in S$ there is an $(n-k)$-space $U_Q$ such that $\cB\cap U_Q=\{P,Q\}$. Thus all points of $S$ are red. As $U_2$ is a $k$-blocking set, $\forall \, Q\in S\colon U_Q\cap U_2=\{Q\}$, so $S\subset U_2$. By interchanging the role of $U_1$ and $U_2$, we see that $U_1$ and $U_2$ both contain at least $3q^k-|\cB|+\theta_{k-1}$ red points. As the same holds for green points as well, we find that $4 (3q^k-|\cB|+\theta_{k-1})\leq |U_1|+|U_2|\leq |\cB|$, thus $|\cB|\geq 2.4q^k+0.8\theta_{k-1}$ in general; if $U_1$ and $U_2$ are $k$-spaces, substituting $|U_1|=|U_2|=\theta_{k}=q^k+\theta_{k-1}$ gives the assertion.
\end{proof}

The next lemma shows under what conditions does Proposition \ref{big} provide a good enough bound on $|\cB|$ to make Proposition \ref{medium} work with $\beta=5$, the value we will typically use.

\begin{lemma}\label{p38works}
Assume $d\leq \alpha q^k$ for some $0\leq \alpha\leq \frac{1}{2}$. Suppose that either
\begin{enumerate}
\item{$k=1$, $q\geq 5$ and $d\leq\min\left\{\frac{q}{10}-2,\,\frac{q(\sqrt{2}-1)}{2} - \frac{9}{2}\right\}$, or}
\item{$k\geq 2$, $q\geq 13$ and $d\leq \frac{q^k}{10}-\frac{9q^{k-1}}{10}-\frac{28q^{k-2}}{10}$.}
\end{enumerate}
Then $|\cB|\leq 2\theta_{k}+2d+2$.
\end{lemma}
\begin{proof}
It is easy to see that the requirement $q > \left(\frac{5}{\sqrt{2}}-2-\alpha-\frac{4}{q}\right)^{-1}>0$ of Proposition \ref{big} holds under $\alpha\leq \frac{1}{2}$ and $q\geq 5$, hence we can conclude that $|\cB|<(4-\sqrt{2})q^k+4\theta_{k-1}+2d+2$. Thus to meet the assumptions of Proposition \ref{medium} with $\beta=5$, it is enough to have
\begin{eqnarray}
q^k&>&9\theta_{k-1}+5(2d+2) \label{medreq1} \mbox{ and }\\
(4-\sqrt{2})q^k+4\theta_{k-1}+2d+2&\leq& 3q^k+\theta_{k-1}-4 \mbox{ or, equivalently,}\\
(\sqrt{2}-1)q^k&\geq&3\theta_{k-1}+2d+6\label{medreq2}.
\end{eqnarray}
For $k=1$, \eqref{medreq1} and \eqref{medreq2} demand $d<\frac{q-19}{10}$ and $d\leq \frac{q(\sqrt{2}-1)}{2}-\frac{9}{2}$.

For $k\geq 2$, using Lemma \ref{lemmata} $a)$ and $c)$ we see that to satisfy \eqref{medreq1} it is enough to have
\[9\left(1+\frac2q\right)q^{k-1}+10\alpha q^k+10\leq q^k,\]
hence, as $k\geq 2$, it is sufficient to require
\[\alpha\leq \frac{1}{10}-\frac{9}{10q}-\frac{28}{10q^2}.\]
Regarding \eqref{medreq2}, we can similarly deduce that
\[(\sqrt{2}-1)q^k\geq 3\left(1+\frac{2}{q}\right)q^{k-1}+2\alpha q^k+6\]
is enough, hence so is
\[ \alpha \leq  \frac{\sqrt{2}-1}{2}-\frac{3}{2q}-\frac{6}{q^2}.\]
It is easy to see that the latter requirement is weaker for $q\geq 9$, so the former one is enough, which is positive if $q\geq 13$. Thus under these conditions Proposition \ref{medium} yields $|\cB|<2\theta_{k}+2d+3$. As the quantities on both sides are integers, the proof is finished.
\end{proof}

\begin{remark}\label{p38worksrem}
If $q\geq 25$, all conditions of Lemma \ref{p38works} are satisfied under $d\leq \frac{q^k}{10}-2q^{k-1}$.
\end{remark}

The considerations so far are enough to prove Theorems \ref{ucnthm} and \ref{ucnstabp}.

\begin{proof}[Proof of Theorem \ref{ucnthm}]
We recall the assumptions $n\geq 3$, $1\leq k<\frac{n}{2}$, $q\geq 17$ if $k=1$ and $q\geq 13$ if $k\geq 2$. Under these, the requirements of Lemma \ref{p38works} are met for $d\leq -1$, $\alpha=0$, so we conclude that $|\cB|\leq 2\theta_{k}$. Result \ref{KM} asserts that $\cB$ contains the union of two $k$-spaces (which are disjoint as $\cB$ is not weighted). Proposition \ref{small} yields that either $|\cB|\geq 2.5q^k+0.5\theta_{k-1}$, a contradiction due to $q$ being large enough, or the coloring is trivial, in which case $d\geq 0$, a contradiction. Thus there is no coloring with $d\leq -1$, in other words, $\UCN(\cH(n,n-1,q))\geq \theta_n-2\theta_{k}+1$. Equality can be reached by trivial colorings since, as $k<\frac{n}{2}$, we can always find two disjoint $k$-spaces, whose union is clearly a $2$-fold $k$-blocking set.
\end{proof}

\begin{proof}[Proof of Theorem \ref{ucnstabp}]
Suppose $d\leq \frac12\left( (\sqrt{2}-1)q^k - 3\theta_{k-1} - 8 \right)$ and $q\geq 11$, $q$ prime. As $11 > \left(\frac{5}{\sqrt{2}}-2-\frac14-\frac{4}{11}\right)^{-1}\approx1.08$, we can apply Proposition \ref{big} with $\alpha=\frac14$ to obtain $|\cB| < (4-\sqrt{2})q^k+4\theta_{k-1}+2d+2 = 3q^k+\theta_{k-1}-6$.
Set $\beta=8$. Then $|\cB|\leq 3q^k+\theta_{k-1}+1-\beta$ and, using $d < \frac{q^k}{4}-\theta_{k-1}-4$, we also have $(\beta-4)q^k>(\beta+4)\theta_{k-1}+\beta(2d+\beta-3)$, so Proposition \ref{medium} applies and yields $|\cB|<2\theta_{k}+2d+6 < 2.5q^k$. Hence, by Theorem \ref{blsetthm}, $\cB$ contains two disjoint $k$-spaces; call them $U_1$ and $U_2$. (Note that this is possible only if $k<\frac{n}{2}$, hence we obtain a contradiction for $k\geq \frac{n}{2}$ showing that no proper coloring satisfies the condition on $d$.) As $|\cB|<2.5q^k$, Proposition \ref{small} claims that our coloring is trivial.

Suppose now that $q$ is not a prime, and recall that our assumptions in this case are $q\geq 25$ and $d\leq \frac12\left(q^{k-1}-\qbinom{k-1}{1}-3\right)$. To apply Lemma \ref{p38works} we need $d\leq \frac{q^k}{10}-2q^{k-1}$, which follows from $d<\frac{q^{k-1}}{2}$ and $q\geq 25$; hence we obtain $|\cB|\leq 2\theta_{k}+2d+2$. The assumed upper bound for $d$ is equivalent to $2d+2\leq (q-2)\frac{q^{k-1}-1}{q-1}$, so we may apply Result \ref{KM} with $t=2$ and $r=q-2$ to see that $\cB$ contains the union of two disjoint $k$-spaces (again, $k \geq \frac{n}{2}$ gives a contradiction). As $|\cB| <  2\theta_{k}+q^{k-1} < 2.5q^k$ clearly holds, Proposition \ref{small} claims that the coloring is trivial.
\end{proof}

\begin{remark}
We do not believe that the upper bound $d\lesssim 0.2q^k$ for the $q$ prime case in the above result is close to be sharp. We think that the limit should be roughly $d\lesssim 0.5q^k$ but to achieve this, one needs to improve Propositions \ref{big} and \ref{medium} significantly, or to use a different approach. Improving only Proposition \ref{big} would allow us to prove the same assertion under $d\lesssim0.25q^k$ (this is the best allowed by Proposition \ref{medium}).
\end{remark}

\subsection{Improvements when $q$ is not a prime: the proof of Theorem \ref{ucnstabq}}

We recall that $\cB=\cB(\cC)$ denotes the union of color classes in the proper coloring $\cC$ with at least two elements, so $\cB$ is a $2$-fold $k$-blocking set in $\PG(n,q)$ colored in a way that each $(n-k)$-space contains at least two points of $\cB$ of the same color. We will rely on some initial assumptions that are essential in the sense that they cannot be changed significantly so that the reasoning still works, which will be derived from the much more restrictive but adjustable requirements of Theorem \ref{ucnstabq} that may be fine-tuned to obtain a similar result. We aim to treat these somewhat separately in order to make future parameter adjustments easier.

\textbf{Initial assumptions:} $q\geq 25$, and $d\leq \frac{q^k}{10}-2q^{k-1}-1$.

As Theorem \ref{ucnstabq} requires $p\geq 11$ and $h\geq 2$, $q = p^h\geq25$ clearly holds. Moreover, the assumption on the number $N$ of colors can be rephrased as an upper bound $d\leq \frac{\gamma}{2}q^k-\theta_{k-1}-\frac32$ with $\gamma=\frac{1}{100}$ (we introduce $\gamma$ as it is, in fact, an adjustable parameter). From this, $d\leq \frac{q^k}{10}-2q^{k-1}-1$ follows if, say, $\gamma\leq \frac{1}{10}$. For a while, let us fix $\gamma\leq \frac{1}{10}$ but leave $\gamma\in\mathbb{R}$ undefined, and require $d\leq \frac{\gamma}{2}q^k-\theta_{k-1}-\frac32$.

The initial assumptions, by Remark \ref{p38worksrem} and Lemma \ref{p38works}, allow us to conclude that $|\cB|\leq 2\theta_{k}+2d+2\leq 2\theta_k+2(\frac{q^k}{10}-2q^{k-1}) \leq  2.2\theta_{k} - 4q^{k-1}$.

Let $\cB'$ denote the unique minimal 2-fold $k$-blocking set contained in $\cB$ (which is the set of essential points for $\cB$, cf.~Result \ref{Harrach}). We want to prove that $\cB'$ is monochromatic; to this end, let us suppose to the contrary that $\cB$ contains a red and a green essential point as well. As $|\cB|\leq 2.2\theta_{k} - 4q^{k-1} < \frac52q^{k}-\frac12$ clearly follows, the $t \mod p$ property (Result \ref{FSSzW}) holds for $\cB'$. We want to show that the coloring $\cC$ is trivial; in other words, $\cB$ contains a monochromatic $2$-fold $k$-blocking set.

We consider three cases depending on the relation between $n$ and $2k$. Our main case is when $n=2k$, in which situation the famous Andr\'e--Bruck--Bose representation of projective planes shall be used to enable us using planar tools. The cases $n>2k$ and $n<2k$ will be traced back to this one in the following way. During this procedure the dimension of the host space, the dimension of the subspaces we want to color properly, the coloring etc.\ may change. We will refer to these modified objects by their original notation equipped with a bar.

\subsubsection{$n\leq2k$}

If $n \leq 2k$ then we simply embed this projective space into $\PG(2k,q)$, and let $\overline{n}=2k$ and $\overline{k}=k$ (i.e., we do nothing if $n=2k$). Color the new points with new and pairwise different colors. After the embedding we get a strict proper coloring $\overline{\cC}$ of $\cH(2k,k,q)=\cH(\overline{n},\overline{n}-\overline{k},q)$ (a $k$-space of $\PG(2k,q)$ intersects the embedded $n$-space containing $\cB$ in a $k+n-2k=n-k$ dimensional subspace). Note that $\cB$ and $\cB'$ are left unchanged, so $|\cB|\leq 2.2\theta_{k} - 4q^{k-1}$ still holds, and from $d\leq \frac{\gamma}{2}q^k-\theta_{k-1}-\frac32$ and $|\cB|\leq 2\theta_{k}+2d+2$ we obtain $|\cB|\leq (2+\gamma)q^k-1\leq 2(q^k+1)+\gamma q^k$ (this last upper bound stands here for future purposes).

\subsubsection{$n>2k$}

Let us embed $\PG(n,q)$ into $\PG(2n-2k,q)$ and let us take an $(n-2k-1)$-space $\cV \subset \PG(2n-2k,q)$ which is disjoint from $\PG(n,q)$ (considered now as a given $n$-space of $\PG(2n-2k,q)$); thus $\PG(2n-2k,q)$ is generated by the original $\PG(n,q)$ and $\cV$. We build a cone $\cK$ upon the base $\cB$ with vertex $\cV$; that is, the cone $\cK$ consists of the points of the lines joining a point $X\in \cB$ with a point $Y \in \cV$.

\begin{lemma}\label{egyszeresmarad}
For an arbitrary point $P \in \PG(2n-2k,q) \setminus (\PG(n,q)\cup\cV)$ there exist a unique pair of points $X\in\PG(n,q)$ and $Y\in\cV$ such that the line defined by $X$ and $Y$ contains $P$.
\end{lemma}
\begin{proof}
If a good pair $X$, $Y$ exists then, clearly, the line $XY$ is contained in $\langle P,\cV\rangle \cap \langle P,\PG(n,q)\rangle$, which is a subspace of dimension $(n-2k-1+1) + (n+1) -(2n-2k) = 1$. Hence a line of this type is unique, and it defines the points $X$ and $Y$ in a unique way.
\end{proof}

The points of $\PG(2n-2k,q)$ not in $\cK$ get pairwise distinct new colors, and let us color the points of $\cK$ in the following way. The points of $\cV$ will get the color of an arbitrarily chosen point of $\cB$, and the points of $\cK\setminus(\cB\cup\cV)$ get the color of their well-defined ancestor (the unique point $X$ in Lemma \ref{egyszeresmarad}) in $\cB$. Finally, let us give weight two to the points of $\cV$. In this way, the coloring of $\cH(2n-2k,n-k,q)$ is proper, since if an $(n-k)$-space $U$ meets $\cV$ then it is blocked by $\cK$ trivially, and if it is skew to $\cV$ then $\langle \cV,U \rangle$ will be an $(2n-3k)$-space such that it meets $\PG(n,q)$ in an $(n-k)$-space $W$, thus $W$ contains two points of $\cB$ of the same color and, by the cone structure, $U$ contains two points of $\cK$ of the same color. Also, the red and the green essential points for $\cB$ in $\PG(n,q)$ remain essential for $\cK$, hence $\cK$ contains a red and a green essential point of weight one.

Let $\overline{n}=2n-2k$, $\overline{k}=n-k$. Note that except for the points of $\cV$, each point of $\cK$ has weight one, and $n-2k-1=\dim(\cV) \leq \frac{2n-2k}{2}-2=\frac{\overline{n}}{2}-2$. Furthermore, the number of points of $\cK$ (with weights) is $|\cB|+2|\cV|+|\cB||\cV|(q-1) = |\cB|+2\theta_{n-2k-1}+|\cB|(q^{n-2k}-1)=|\cB|q^{n-2k}+2\theta_{n-2k-1}$. On the one hand, as $|\cB|\leq 2\theta_k+2d+2 \leq 2\theta_k + 2(\frac{q^k}{10}-2q^{k-1})$, we obtain $|\cK|\leq 2(\theta_{n-k}-\theta_{n-2k-1}) + 2(\frac{q^{n-k}}{10}-2q^{n-k-1}) + 2\theta_{n-2k-1} = 2\theta_{\overline{k}}+2(\frac{q^{\overline{k}}}{10}-2q^{\overline{k}-1})\leq 2.2\theta_{\overline{k}}-4q^{\overline{k}-1}$; on the other hand, from the assumption $d\leq \frac{\gamma}{2}q^k-\theta_{k-1}-\frac32$ we get $|\cB|\leq 2\theta_k+2d+2\leq (2+\gamma)q^k-1$, whence $|\cK|\leq (2+\gamma)q^{n-k}-q^{n-2k}+2\theta_{n-2k-1}\leq (2+\gamma)q^{n-k} \leq 2(q^{\overline{k}}+1)+\gamma q^{\overline{k}}$.

\subsubsection{The main case, $n=2k$}

In both of the above cases, we ended up in a projective space of order $\overline{n}=2\overline{k}$ whose points admit a proper coloring with respect to $\overline{k}$-spaces, and the union of the color classes of size at least two form a $2$-fold $\overline{k}$-blocking set $\overline{\cB}$ of size at most $|\overline{\cB}|\leq 2.2\theta_{\overline{k}} - 4q^{\overline{k}-1}$ on the one hand, and $|\overline{\cB}|\leq 2(q^{\overline{k}}+1)+\gamma q^{\overline{k}}$ on the other hand. Moreover, $\overline{\cB}$ is either non-weighted, or the set of points with weight more than one is a subplane of dimension at most $\overline{k}-2$, and all points in this subplane are of weight two. In both cases, our indirect assumption assures that there exist red and green essential points of weight one. From now on we will work in this setting only, so we reset the notation and omit the bars.

For future purposes, we need to find a hyperplane $H$ that intersects $\cB$ in at most $2.2\theta_{k-1}$ points and contains all points of weight two (if there is any). If $k=1$, we are done (otherwise $\cB$ blocks every line of $\PG(2,q)$ at least three times, so $|\cB|\geq 3(q+1)$, a contradiction). Suppose now $k\geq 2$, and recall $|\cB|\leq 2.2\theta_k - 4q^{k-1}$. Let $U_{-2}$ be the $(k-2)$-space consisting of the points of weight two or, if there are no such points, an arbitrary $(k-2)$-space. Among the $\theta_{k+1}$ distinct $(k-1)$-spaces containing $U_{-2}$, there must be one, say, $U_{-1}$, that contains no point of $\cB\setminus U_{-2}$, otherwise $|\cB|\geq \theta_{k+1}>2.2\theta_{k}$, a contradiction. Among the $\theta_k$ distinct $k$-spaces containing $U_{-1}$ there must be one, say, $U_0$, that contains at most two points of $\cB\setminus U_{-1}$, otherwise $|\cB|\geq 3\theta_k>2.2\theta_k$. Suppose now that the $(k+i)$-space $U_i$ contains at most $2.2q^i$ points of $\cB\setminus U_{i-1}$ ($0\leq i\leq k-3$). Then among the $\theta_{k-i-1}$ distinct $(k+i+1)$-spaces containing $U_i$, there must be one, say, $U_{i+1}$, that contains at most $2.2q^{i+1}$ points of $\cB\setminus U_{i}$, otherwise $|\cB|> 2.2q^{i+1}\theta_{k-i-1}=2.2\theta_k-2.2\theta_i > 2.2\theta_{k}-2.2\theta_{k-2} > 2.2\theta_k - 4q^{k-1} \geq |\cB|$, a contradiction.
To find an appropriate hyperplane $U_{k-1}$, we claim that among the $\theta_1=q+1$ distinct $(2k-1)$-spaces containing $U_{k-2}$ there is one that contains at most $2.2q^{k-1}-2\theta_{k-2}$ points of $\cB\setminus U_{k-2}$, otherwise $|\cB| >  \left(2.2q^{k-1}-2\theta_{k-2}\right)(q+1) = 2.2q^k + 2.2q^{k-1} -2(\theta_{k-1}-1)-2\theta_{k-2} = 2.2\theta_k - 2 \theta^{k-1} -4.2\theta_{k-2} + 2 = 2.2\theta_k -2q^{k-1} - 6.2\frac{q^{k-1}-1}{q-1} +2 > 2.2\theta_k - 4q^{k-1} \geq |\cB|$, a contradiction (where we use $q\geq 25$).
Thus we find an $(n-1)$-space $U_{k-1}$ such that
\[
|\cB\cap U_{k-1}| = |\cB\cap U_{-2}| + \sum_{i=0}^{k-1}|(\cB\setminus U_{i})\cap U_{i+1}| \leq |\cB\cap U_{-2}| + 2.2\theta_{k-2} +
\]
\[
+ 2.2q^{k-1} - 2\theta_{k-2} \le 2\theta_{k-2} + 2.2\theta_{k-1} - 2\theta_{k-2} = 2.2\theta_{k-1}.
\]
We set $H=U_{k-1}$ to be the hyperplane (a $(2k-1)$-space) admitting the properties claimed. Andr\'e \cite{Andre} and independently Bruck and Bose \cite{BB1,BB2} developed a method, the well-known \emph{Andr\'e-Bruck-Bose representation}, for representing translation planes of order $q^{h}$ with kernel containing $\GF(q)$ in the projective space $\PG(2h,q)$. It arises from a suitable $(h-1)$-spread of the hyperplane at infinity in $\PG(2h,q)$. The affine lines of the plane are $h$-dimensional subspaces containing the $(h-1)$-spaces of the $(h-1)$-spread. The ideal points correspond to the elements of the spread. Thus a point set intersecting every $h$-space yields a blocking set in the plane $\PG(2,q^h)$.

It is well-known that an arbitrary $(k-1)$-space can be mapped to any other $(k-1)$-space with a suitable linear transformation. By the previous observations we can take a $(k-1)$-spread $\cS$ of $H$ (i.e., a set of $(k-1)$-spaces that partition $H$) in such a way that if $\cV$ exist it will be contained in one of the spread elements. Moreover, this transitivity property allows us to choose such a Desarguesian (also called regular) spread, too.

Remember that we have already assumed on the contrary that $\cB'$, the minimal part of $\cB$, contains red and green essential points of weight one. By using Lemma \ref{esspt} and the choice of $H$ one can see that $\cB'$ must have both red and green affine points. In the following we will show that the minimal part of $\cB$ must be monochromatic which will give us a contradiction.

Let us define a point-line incidence structure $\Pi=\Pi(H,\cS)$ in the following way:
\begin{itemize}
\item the points of $\Pi$ are the points of $\PG(2k,q)\setminus H$ and the elements of $\cS$;
\item for each $k$-dimensional subspace $U$ of $\PG(2k,q)$ such that $U\cap H\in\cS$, the set $(U\setminus H)\cup \{U\cap H\}$ is considered to be a line of $\Pi$, as well as $\cS$;
\item a point is incident with a line if it is an element of it.
\end{itemize}
Then $\Pi$ is well-known to be a projective plane of order $\wt{q}:=q^k$ by the Andr\'e-Bruck-Bose representation, and since $\cS$ is a Desarguesian spread, then $\Pi\simeq\PG(2,\wt{q})$.

We will consider $\cS$ as the line at infinity in $\Pi$, and a point of $\Pi$ is called ideal or affine according to whether it is on the ideal line or not.

\begin{definition}\label{wtcoloring}
From the coloring $\cC$ of $\PG(2k,q)$, we define a coloring $\wt{\cC}$ of the points of $\Pi$ in the following way.
\begin{itemize}
\item For an affine point $P$ of $\Pi$, let $P$ inherit its color naturally from the coloring $\cC$.
\item For an ideal point $S\in\cS$, we distinguish two cases. On the one hand, if each point of $S$ forms a singleton color class of $\cC$ (i.e., $\cB \cap S = \emptyset$), then let the color of $S$ be the color of an arbitrarily chosen point of $S$. On the other hand, if there is a color class of $\cC$ of size at least two containing a point of $S$ (i.e., $\cB \cap S \neq \emptyset$), then color $S$ with a color $i$ such that $C_i\cap S\neq \emptyset$, $|C_i|\geq 2$, and for all $j\in\{1,2,\ldots, N\}$ we have $|C_i\cap S |\geq | C_j\cap S |$.
\end{itemize}
\end{definition}

Note that $\wt{\cC}$ is an $N$-coloring of $\Pi$ that might not be strict.

\begin{definition}\label{wtset}
From a weighted point set $X$ of $\PG(2k,q)$ with weight function $w_X$, we define a weight function $\wt{w}_{X}$ on the points of $\Pi$ in the following way.
\begin{itemize}
\item For an affine point $P$ of $\Pi$, let $\wt{w}_{X}(P)=w_X(P)$ if $P\in X$ and $\wt{w}_{X}(P)=0$ otherwise.
\item For an ideal point $S\in\cS$, let $\wt{w}_{X}(S)=|S\cap X|$ (counted with weights; that is, $\wt{w}_X(S)=\sum_{P\in S}w_X(P)$).
\end{itemize}
For the point set $X$, $\wt{X}$ denotes the weighted point set of $\Pi$ corresponding to the weight function $\wt{w}_X$ (zero weight points are not considered as elements of $\wt{X}$).
\end{definition}

Consider now $\wt{\cB}=\wt{\cB(\cC)}$ (recall that $\cB$ may be weighted). If $S\notin\wt{\cB}$, then the points of $S$ have pairwise distinct colors in $\cC$ (and all are singletons). If $S\in\wt{\cB}$ is of weight one, then the color of $S$ at $\wt{\cC}$ is the same as the color of the unique point in $S\cap\cB$ at $\cC$. We remark that for the union $\cB(\wt{\cC})$ of color classes of size at least two of $\wt{\cC}$ in $\Pi$, $\cB(\wt{\cC})\subseteq \wt{\cB}$, but equality does not follow immediately from the definitions. In the sequel, we will work with $\wt{\cB}$ using the property that every line of $\Pi$ intersects it in at least two equicolored points, yet we make the following, slightly stronger observation.

\begin{proposition}\label{wtproper}
The coloring $\wt{\cC}$ with the weight function $\wt{w}_{\cB}$ is a proper weighted coloring of $\Pi$; that is, every line of $\Pi$ contains some points of the same color whose weights add up to at least two.
\end{proposition}
\begin{proof}
Let $U$ be the $k$-space of $\PG(2k,q)$ corresponding to a line $\ell$ of $\Pi$, and let $S=U\cap H$. If $\wt{w}_{\cB}(S)\geq2$, we are done. If $\wt{w}_{\cB}(S)\leq1$ (whence $\{P,Q\}\not\subset S$ follows), then, as $\cC$ is proper, $U$ contains two distinct points of the same color with respect to $\cC$, say, $P$ and $Q$; note that $\{P,Q\}\subset\cB$. If both $P$ and $Q$ are affine points (which is the case if $\wt{w}_\cB(S)=0$), then we are also done. Suppose now $P\in S$, $Q\notin S$ and $\wt{w}(S)=1$. Then, as $P$ is the unique point of $S\cap\cB$, the color of $S$ at $\wt{\cC}$ is the same as the color of $P$ at $\cC$, and so $S$ and $Q$ are two points of $\ell$ having the same color at $\wt{\cC}$.
\end{proof}

It is clear from Definition \ref{wtset} that $\wt{\cB}$ and $\wt{\cB'}$ (that is, the weighted point set in $\Pi$ obtained from $\cB'$) are weighted double blocking sets in $\Pi$ of size (total weight) $|\wt{\cB}|=|\cB|$ and $|\wt{\cB'}|=|\cB'|$; however, $\wt{\cB'}$ may not be minimal. Let $\wh{\cB}$ be the unique minimal weighted double blocking set contained in $\wt{\cB}$ (cf.\ Result \ref{Harrach}); then $\wh{\cB}\subset\wt{\cB'}$ follows.

\begin{proposition}\label{traceback}
If $\wh{\cB}$ is monochromatic at $\wt{\cC}$, then $\cC$ is trivial.
\end{proposition}
\begin{proof}
Clearly, $|\wh{\cB}|\geq 2(\wt{q}+1)=2q^k+2$. Suppose that each point of $\wh{\cB}$ is, say, green at $\wt{\cC}$. As $\wh{\cB}$ is minimal, each ideal point $S\in\wh{\cB}$ has weight at most two (the affine points of $\wh{\cB}$ have weight exactly one). An ideal point $S\in \wh{\cB}$ as a $(k-1)$-dimensional subspace in $\PG(2k,q)$ must contain at least one green point (with respect to $\cC$). By the choice of $H$ we know that $|\cB' \cap H| \leq 2.2\theta_{k-1}$. Therefore $\cB'$ contains at least $|\wh{\cB}|- \frac{2.2\theta_{k-1}}{2} \ge 2q^k + 2 - 1.1\theta_{k-1}$ green points.
We recall our general assumptions $|\cB'|\leq|\cB|\leq 2.2\theta_k-4q^{k-1}$ and suppose to the contrary that $\cC$ is not trivial. Then $\cB'$ contains a point that is not green but, say, red. As $\cB'$ is minimal, this red point is essential and simple thus Lemma \ref{esspt} claims that the number of red simple points is more than $3q^k-|\cB'|$, whenceforth $|\cB'|>2q^k + 2 - 1.1\theta_{k-1} + 3q^k-|\cB'|$, that is, $|\cB'|\geq 2.5q^k -0.55\theta_{k-1}$ follows, a contradiction. Hence $\cB'$ is all green, thus $\cC$ is trivial.
\end{proof}

By Proposition \ref{traceback}, it is enough to show that $\wh{\cB}$ is monochromatic. We will do this along the same main ideas as in \cite[Proposition 3.14]{BHSz}; however, the ideas must have been adapted to the presence of weights. We need the following lemma.

\begin{lemma}\label{affinesspt}
Let $P\in\cB'\setminus H$. Then $P$ is essential for $\wt{\cB'}$ in $\Pi$; consequently, $P\in\wh{\cB}$.
\end{lemma}
\begin{proof}
Suppose to the contrary. Then every line of $\Pi$ through $P$ intersects $\wt{\cB'}$ in at least three points (with respect to $\wt{w}_{\cB'}$). This yields that for every $S\in\cS$, the $k$-space $\langle P,S\rangle$ of $\PG(2k,q)$ intersects $\cB'$ in at least three and thus, by Result \ref{FSSzW}, in at least $p+2$ points. As the $q^k+1$ distinct $k$-spaces of form $\langle P,S\rangle$, $S\in\cS$, pairwise intersect in $P$ only, we get $2.5q^k\geq |\cB|\geq|\cB'|\geq (p+1)(q^k+1)+1$, a contradiction.
\end{proof}

By our indirect assumption we know that $\cB'$ contains both red and green simple affine points. By Lemma \ref{affinesspt} and the definition of $\wt{\cC}$ we see that $\wh{\cB}$ also contains both red and green affine (and hence single) points.

If there are other color classes in $\cC$ containing more than two points, replace their color by red. In this way we obtain a nontrivial proper coloring $\cC'$ such that $\wt{\cB}(\cC)=\wt{\cB}(\cC')$; thus it is enough to restrict our attention for colorings admitting only two color classes of size more than one. Then the points of $\wh{\cB}$ are also either red or green and both colors actually occur in the affine part. Denote the set of red points of $\wh{\cB}$ by $\wh{\cB}_r$ and the set of green ones by $\wh{\cB}_g$.

Recall that $|\cB|=|\wt{\cB}| \le 2(\wt{q}+1)+\gamma \wt{q}$ for some $\gamma\leq\frac{1}{10}$, and let us write $X=\gamma \wt{q}$. By Result \ref{FSSzW}, every line meets $\wh{\cB}$ in $2 \pmod {p^e}$ points, where $e\geq 1$ is the largest integer for which this property holds. Write $|\wh{\cB}| = 2(\wt{q}+1)+c$. Note that since $\wh{\cB}$ is a minimal double blocking set in $\Pi$, every point of it has weight at most two; moreover, by its definition, all double points of $\wh{\cB}$ are on the ideal line. It is easy to see that if $P \in \wh{\cB}$ is a single point, then there are at least $\wt{q}+1-\frac{\wt{q}+c}{p^e}$ bisecants through it and if $P$ is a double point, then there are at least $\wt{q}+1-\frac{2\wt{q}+c}{p^e}$ bisecants through it. 
If $P\in\wh{\cB}$ is an affine single point, then at least $\wt{q}+1 - \frac{\wt{q}+c}{p^e} - (X-c)$ of the bisecants through $P$ to $\wh{\cB}$ are bisecants to $\wt{\cB}$ as well.
Since $\wt{\cC}$ is proper, the points on these bisecants must have the same color as $P$. As there are both red and green affine (and hence single) points of $\wh{\cB}$, we find
\begin{eqnarray}
|\wh{\cB}_r| &\ge&  (\wt{q}+2) - \frac{\wt{q}+c}{p^e} - (X-c), \label{redbiseclower}\\
|\wh{\cB}_g| &\ge&  (\wt{q}+2) - \frac{\wt{q}+c}{p^e} - (X-c) , \label{greenbiseclower}
\end{eqnarray}
which also immediately gives
\begin{equation}
|\wh{\cB}_r| = |\wh{\cB}| - |\wh{\cB}_g| \le 2(\wt{q}+1)+c - \left ( \wt{q}+2 - \frac{\wt{q}+c}{p^e} - (X-c) \right ) = \wt{q} + \frac{\wt{q}+c}{p^e} + X. \label{redbisecupper}
\end{equation}

Our aim now is to show that one of the color classes, say, the red class, contains even more points than what was shown above, leading to a lower bound on $|\wh{\cB}|$ large enough to get a contradiction. To this end we want to find an affine single red point in $\wh{\cB}$ that has many non-bisecant lines through it on which there are more red points than green.

For a line $\ell$ of $\Pi$, let $n_\ell = |\ell\cap \wh{\cB}|$, $n_\ell^r = |\ell \cap \wh{\cB}_r|$, $n_\ell^g = |\ell \cap \wh{\cB}_g|$.
Clearly, $n_\ell^r + n_\ell^g = n_\ell$ holds for all line $\ell$.
We denote the affine part of $\wh{\cB}$ by $\wh{\cB}^a$ and for a line $\ell$ different from $\cS$, define $\bar{n}_\ell$, $\bar{n}_\ell^r$, $\bar{n}_\ell^g$ similarly as above but with respect to $\wh{\cB}^a$. Again, $\bar{n}_\ell^r + \bar{n}_\ell^g = \bar{n}_\ell$ holds for every affine line $\ell$. Clearly, $n_\ell -2 \le \bar{n}_\ell \le n_\ell$ also holds. We recall Result \ref{FSSzW} and $p^e \ge 3$. Observe that if $\bar{n}_\ell = 0$, then $\ell$ must meet the ideal line in a double point of $\wh{\cB}$; if $\bar{n}_\ell = 1$, then $\ell$ must meet the ideal line in a single point of $\wh{\cB}$; and if $\bar{n}_\ell = 2$, then $\ell$ must meet the ideal line outside of $\wh{\cB}$. Also, $\bar{n}_\ell>2\Leftrightarrow n_\ell>2$. Let us denote the set of single and double points of $\cS$ by $\cS^1$ and $\cS^2$, respectively.
With these notations one can find the inequalities
    \[
    \sum_{\ell \in \cL \setminus \ell_{\infty},~\bar{n}_\ell = 1} \bar{n}_\ell \le |\cS^1|\wt{q} \qquad \mbox{and} \qquad \sum_{\ell \in \cL \setminus \ell_{\infty},~\bar{n}_\ell = 2} \bar{n}_\ell \le 2 \cdot (\wt{q}+1 - |\cS^1|- |\cS^2|)\wt{q}.
    \]
Clearly, we have $\sum_{\ell \in \cL \setminus \ell_{\infty}} \bar{n}_\ell = |\wh{\cB}^a| \cdot (\wt{q}+1)$. Let $\Delta=|\cS^1| + 2|\cS^2|=|\wh{\cB}| - |\wh{\cB}^a|$. Then
    \[
    \sum_{\ell \in \cL \setminus \ell_{\infty},~n_\ell > 2} \bar{n}_\ell = \sum_{\ell \in \cL \setminus \ell_{\infty},~\bar{n}_\ell > 2} \bar{n}_\ell = \sum_{\ell \in \cL \setminus \ell_{\infty}} \bar{n}_\ell - \left ( \sum_{\ell \in \cL \setminus \ell_{\infty},~\bar{n}_\ell = 1} \bar{n}_\ell + \sum_{\ell \in \cL \setminus \ell_{\infty},~\bar{n}_\ell = 2} \bar{n}_\ell \right ) \ge
    \]
    \[
    \ge \sum_{\ell \in \cL \setminus \ell_{\infty}} \bar{n}_\ell - \wt{q}(2\wt{q} + 2 - \Delta) = (2(\wt{q}+1) + c - \Delta)(\wt{q}+1) - \wt{q}(2\wt{q} + 2 - \Delta) =
    \]
    \[
    = (c+2)\wt{q} + (c+2 - \Delta).
    \]
    We will refer to a line $\ell$ as a long secant if $n_\ell > 2$ holds. Let $\cL^r$ be the set of affine long secants with $\bar{n}_\ell^r > \bar{n}_\ell^g$; define $\cL^g$ and $\cL^{=}$ analogously. Without loss of generality we may assume that $\sum_{\ell \in \cL^r} \bar{n}_\ell^r \ge \sum_{\ell \in \cL^g} \bar{n}_\ell^g$, therefore
    \[
    (c+2)\wt{q} + (c+2 - \Delta) \le \sum_{\ell \in \cL \setminus \ell_{\infty},~n_\ell > 2} \bar{n}_\ell = \sum_{\ell \in \cL^r} (\bar{n}_\ell^r + \bar{n}_\ell^g) + \sum_{\ell \in \cL^g} (\bar{n}_\ell^r + \bar{n}_\ell^g) + \sum_{\ell \in \cL^{=}} (\bar{n}_\ell^r + \bar{n}_\ell^g) \le
    \]
    \[
    \le \sum_{\ell \in \cL^r} 2\bar{n}_\ell^r + \sum_{\ell \in \cL^g} 2\bar{n}_\ell^g + \sum_{\ell \in \cL^=} 2\bar{n}_\ell^r \le \sum_{\ell \in \cL^r \cup \cL^=} 4\bar{n}_\ell^r.
    \]
    We call an affine long secant $\ell$ with $\bar{n}_\ell^r \ge \bar{n}_\ell^g$ an almost red line. From the last inequality we get that there is a red affine point $P\in \wh{\cB}_r$ such that the number of almost red lines through $P$ is at least
\[\frac{(c+2)\wt{q} + (c+2 - \Delta)}{4|\wh{\cB}_r^a|} \ge \frac{(c+2)\wt{q} + (c+2 - \Delta)}{4|\wh{\cB}_r|},\] where $\wh{\cB}_r^a$ is the affine part of $\wh{\cB}_r$.

    By the $t \pmod p$ result (Result \ref{FSSzW}) we know that for a long secant $\ell$, $n_\ell \ge p^e + 2$ and since $n_\ell - 2 \le \bar{n}_\ell$, we can deduce that $\bar{n}_\ell^r + \bar{n}_\ell^g \ge p^e$ for all long secants. Hence, since $P$ is an affine single red point, any almost red long secant through $P$ contains at least $\frac{p^e}{2} - 1$ red points of $\wh{\cB}$ different from $P$ on it (counted with weights). Taking into account the number of red points located on the bisecants through $P$ \eqref{redbiseclower} and also the upper bound \eqref{redbisecupper} on $\wh{\cB}_r$, this yields
    \[
    (\wt{q}+2) - \frac{\wt{q}+c}{p^e} - (X-c) + \frac{(c+2)\wt{q} + (c+2 - \Delta)}{4 |\wh{\cB}_r|} \cdot \left ( \frac{p^e}{2}-1 \right ) \leq |\wh{\cB}_r| \leq \wt{q} + \frac{\wt{q}+c}{p^e} + X.
    \]
Rearranging the above inequality gives
    \[
    2X + \frac{2(\wt{q}+c)}{p^e} - c-2 \ge \frac{(c+2)\wt{q} + (c+2 - \Delta)}{4 |\wh{\cB}_r|} \cdot \left ( \frac{p^e}{2}-1 \right ).
    \]
From Result \ref{FSSzW} we know that $c+2\ge \frac{\wt{q}}{p^e+1} - 1$. Applying this and also $|\wh{\cB}_r|\leq \wt{q} + \frac{\wt{q}+c}{p^e} + X$ we get
    \[
    2X + \frac{2(\wt{q}+c)}{p^e} - \frac{\wt{q}}{p^e + 1} + 1 \ge \frac{\left ( \frac{\wt{q}}{p^e+1} - 1 \right )\wt{q} + \left ( \frac{\wt{q}}{p^e+1} - 1 -\Delta \right )}{4 \cdot \left ( \wt{q} + \frac{\wt{q}+c}{p^e} + X \right )} \cdot \left ( \frac{p^e}{2}-1 \right )
    \]

Let us write $X=\gamma\wt{q}$ again. Multiplying both sides with the whole denominator of the right (note that this is surely a positive number) and arranging everything to the left side we get the following due to a lengthy computation:

\begin{equation}\label{4soros}
\begin{gathered}
\wt{q}^2 p^e \left (-1 + 16\gamma + 16\gamma^2 \right ) + 2 \wt{q}p^{2e} + \wt{q}^2 \left ( 10 + 40\gamma + 16\gamma^2 \right ) + \frac{\wt{q}^2}{p^e} \left ( 24 + 32\gamma \right ) +\\
\wt{q}c \left ( 16 + 32\gamma \right ) + 16\frac{\wt{q}^2}{p^e} + \wt{q}p^e \left ( 6 + 8\gamma \right ) + \frac{\wt{q}c}{p^e} \left ( 40 + 32\gamma \right ) +  \wt{q}\left(16 + 8\gamma\right) + \\
32 \frac{\wt{q}c}{p^{2e}} + 16 \frac{c^2}{p^{2e}} + 16\frac{c^2}{p^e} + 8c + 8 \frac{\wt{q}}{p^e} + 8 \frac{c}{p^e} - 2 - 2\Delta + 2p^{2e} + 2\Delta p^{2e} \ge 0
\end{gathered}
\end{equation}

If $p^e = \wt{q}$, then every line which is not a $2$-secant to $\wh{\cB}$ is contained completely in $\wh{\cB}$ (and the ideal point of it has weight two) since the affine points are single ones and an ideal point has weight at most two. Hence if there exists a double point in the ideal line then $\wh{\cB}$ has to be the union of two complete lines and otherwise every line is a $2$-secant to $\wh{\cB}$. In the first case we get a contradiction with Proposition \ref{small} and in the latter case we get that the number of lines has to be equal to $\binom{|\wh{\cB}|}{2}$, but now $|\wh{\cB}| = 2(\wt{q}+1)$. Hence $p^e = \wt{q}$ is not possible.

If $p^e < \wt{q}$, then the leading term in expression (\ref{4soros}) is $\wt{q}^2 p^e \left (-1 + 16\gamma + 16\gamma^2 \right )$. If $\gamma$ is chosen so that $-1 + 16\gamma + 16\gamma^2 <0$ and $q$ and $p^e$ are large enough, then the leading term overflows the remaining ones, hence we will get a contradiction and conclude that our coloring $\cC$ must be trivial. The coefficient is negative if $\frac{1}{\gamma}>8+4\sqrt{5}\approx 16,944$ but then the remaining terms can be quite large. Thus at this point we make a rather arbitrary choice of the parameters in our likes and, in case someone would need a differently set result, we will make a remark on the other possible choices.

Let us consider the non-negative expression on the left side of (\ref{4soros}) as a function $f=f(\wt{q},p^e,\gamma,c,\Delta)$. Clearly, $f$ is increasing in $c$, $\Delta$ and $\gamma$. By the definitions of $\Delta$ and $c$ we immediately see that $\Delta\leq 2(\wt{q}+1)$ and $c\leq X=\gamma\wt{q}$, thus $g(\wt{q},p^e,\gamma):=f(\wt{q},p^e,\gamma,\gamma\wt{q},2(\wt{q}+1))\geq 0$ follows. Let us now fix the value of $\gamma=\frac{1}{100}$. Now $p^{2e}\cdot g\left (\wt{q},p^e,\frac{1}{100}\right )=$
\[
(6\wt{q} + 6)p^{4e} + \left ( - \frac{524}{625}\wt{q}^2 + \frac{152}{25}\wt{q} \right ) p^{3e} + \left (\frac{6603}{625}\wt{q}^2 + \frac{304}{25}\wt{q} - 6 \right )p^{2e} +
\left ( \frac{25453}{625}\wt{q}^2 + \frac{202}{25}\wt{q}\right )p^e + \frac{201}{625}\wt{q}^2 \ge 0.
\]

Since $p^e \ne \wt{q}$, we know that $p^e \le \frac{\wt{q}}{p}$ holds, and on the other hand, from Result \ref{FSSzW} and $|\wt{B}|\leq 2(\wt{q}+1)+\gamma\wt{q}$ one can deduce that $p^e \ge \frac{\wt{q}}{\gamma q + 3} - 1$, which is equivalent to $p^e \ge 99 - \frac{30000}{\wt{q}+300}$. Since Theorem \ref{ucnstabq} requires $\wt{q}\ge239$ and $p\ge 11$, we may increase the terms with positive coefficients by changing $p^e$ to $\frac{\wt{q}}{11}$ or by multiplying with $\frac{\wt{q}}{239}$. Moreover, we can decrease the terms with negative coefficients by changing $p^e$ to $47$, since $\wt{q}\geq 239$ implies $p^e \ge 99 - \frac{30000}{239+300} \approx 43,341$, whence $p^e\geq 47$. With these three elementary observations one can give an upper bound $p^{2e}\cdot g(\wt{q},p^e,\frac{1}{100})\le$
\small
\[
\begin{aligned}
&(6\wt{q}+6)p^{3e}\cdot\frac{\wt{q}}{11} + \left (-\frac{524}{625}\wt{q}^2 + \frac{152}{25}\wt{q} \right )p^{3e} + \left ( \frac{6603}{625}\wt{q}^2 + \frac{304}{25}\wt{q} - 6 \right )p^{2e} + \left ( \frac{25453}{625}\wt{q}^2 + \frac{202}{25}\wt{q} \right )p^e + \frac{201}{625}\wt{q}^2\\
&=\left (-\frac{2014}{6875} \right )p^{3e}\wt{q}^2 + \frac{1822}{275}p^{3e}\wt{q} + \frac{6603}{625}p^{2e}\wt{q}^2 + \frac{304}{25}p^{2e}\wt{q} - 6p^{2e} + \frac{25453}{625}p^e\wt{q}^2 + \frac{202}{25}p^e\wt{q} + \frac{201}{625}\wt{q}^2\\
& \le \left (-\frac{2014}{6875} \right )p^{3e}\wt{q}^2 + \frac{1822}{275}p^{3e}\frac{\wt{q}^2}{239} +
\frac{6603}{625}p^{2e}\wt{q}^2 + \frac{304}{25}p^{2e}\wt{q} - 6p^{2e} + \frac{25453}{625}p^e\wt{q}^2 + \frac{202}{25}p^e\wt{q} + \frac{201}{625}\wt{q}^2\\
&= \left ( -\frac{435796}{1643125}\right )p^{3e}\wt{q}^2 + \frac{6603}{625}p^{2e}\wt{q}^2 + \frac{304}{25}p^{2e}\wt{q} - 6p^{2e} + \frac{25453}{625}p^e\wt{q}^2 + \frac{202}{25}p^e\wt{q} + \frac{201}{625}\wt{q}^2\\
&\le \left ( -\frac{435796}{1643125}\right )p^{2e}\wt{q}^2\cdot 47 + \frac{6603}{625}p^{2e}\wt{q}^2 + \frac{304}{25}p^{2e}\wt{q} - 6p^{2e} + \frac{25453}{625}p^e\wt{q}^2 + \frac{202}{25}p^e\wt{q} + \frac{201}{625}\wt{q}^2\\
&= \left (- \frac{4997}{2629}\right )p^{2e}\wt{q}^2 + \frac{304}{25}p^{2e}\wt{q} - 6p^{2e} + \frac{25453}{625}p^e\wt{q}^2 + \frac{202}{25}p^e\wt{q} + \frac{201}{625}\wt{q}^2\\
&\le \left (- \frac{4997}{2629}\right )p^e\wt{q}^2\cdot47 + \frac{304}{25}p^e\frac{\wt{q}^2}{11} - 6p^{2e} +
\frac{25453}{625}p^e\wt{q}^2 + \frac{202}{25}\frac{\wt{q}^2}{11} + \frac{201}{625}\wt{q}^2\\
&= \left (-\frac{78054538}{1643125}\right )p^e\wt{q}^2 - 6p^{2e} + \frac{7261}{6875}\wt{q}^2
\le \left (-\frac{78054538}{1643125}\right )\cdot 47\wt{q}^2 - 6p^{2e} + \frac{7261}{6875}\wt{q}^2\\
& = - \frac{3666827907}{1643125}\wt{q}^2 - 6p^{2e} < 0,
\end{aligned}
\]

\normalsize

\noindent which is a contradiction, hence the coloring must be trivial. Thus we finished the proof of Theorem \ref{ucnstabq}. \hfill $\Box$

\vspace{8mm}

\begin{remark}\label{ucnthm:remark}
One may want to choose a suitable $\gamma$ to obtain a different stability gap in Theorem \ref{ucnstabq}, in which case the required lower bounds on $q^k$ and $p$ must be adjusted appropriately. The main limitation in our proof is that $\frac{1}{\gamma}>8+4\sqrt{5}\approx 16.944$ must hold. In this way, one may get the conclusion of Theorem \ref{ucnstabq} under the conditions $\delta\leq \frac{1}{40}q^k-\theta_{k-1}-\frac32$ (that is $\gamma=\frac{1}{20}$) and $p\geq 151$ (here the automatic lower bound $q^k\geq p^2$ is enough), or $\delta\leq \frac{1}{100}q^k-\theta_{k-1}-\frac32$ (that is, $\gamma=\frac{1}{50}$), $p\geq 17$ and $q^k\geq479$, for example.
\end{remark}

\section*{Acknowledgement}

The authors gratefully acknowledge the support of the bilateral Slovenian--Hungarian Joint Research Project no.\ NN 114614 (in Hungary) and N1-0032 (in Slovenia). The second author was also supported by the J\'anos Bolyai Research Grant.

\begin{flushleft}
Zolt\'an L.\ Bl\'azsik, Tam\'{a}s H\'{e}ger\\
MTA--ELTE Geometric and Algebraic Combinatorics Research Group, and\\
ELTE E\"otv\"os Lor\'and University, Budapest\\
1117 Budapest, P\'azm\'any P.\ stny.\ 1/C, Hungary\\
Department of Computer Science\\
e-mail: {\sf blazsik@caesar.elte.hu}, {\sf heger@caesar.elte.hu}\\
\end{flushleft}

\begin{flushleft}
Tam\'as Sz\H{o}nyi\\
ELTE E\"otv\"os Lor\'and University, Budapest, Hungary\\
Department of Computer Science, and\\
MTA--ELTE Geometric and Algebraic Combinatorics Research Group\\
1117 Budapest, P\'azm\'any P.\ stny.\ 1/C, Hungary, and\\
UP FAMNIT, University of Primorska, Glagolja\v{s}ka 8, 6000, Koper, Slovenia\\
e-mail: {\sf szonyi@cs.elte.hu}
\end{flushleft}

\end{document}